\documentclass[final]{siamltex}  % Comment this line out
\usepackage{cite}
\usepackage{graphics} % for pdf, bitmapped graphics files
\usepackage{amsmath} % assumes amsmath package installed
\usepackage{amssymb}  % assumes amsmath package installed
\usepackage{subfigure}

\usepackage{scalefnt}

\usepackage{color}
\usepackage{hyperref}
\usepackage{enumerate}

\newtheorem{example}{Example}

\newcommand{\re}{{\mathbb R}}

\newcommand{\cA}{{\cal{A}}}

\newcommand{\cM}{{\cal{M}}}

\newcommand{\inter}{{{\rm int} }}

\providecommand{\com[2]}{\begin{tt}[#1: #2]\end{tt}}

\providecommand{\rmj[1]}{{\color{black}#1}}
\title{\large \bf SOS-Convex Lyapunov Functions and\\ Stability of Difference Inclusions}
\author{Amir Ali Ahmadi and Rapha\"el~M. Jungers
\thanks{Amir Ali Ahmadi is with the Department of Operations Research and Financial Engineering at Princeton University (email: \texttt{a\_a\_a@princeton.edu}). His research has been partially supported by the DARPA Young Faculty Award, the Young Investigator Award of the AFOSR, the CAREER Award of the NSF, the Google Faculty Award, and the Sloan Fellowship. Rapha\"el Jungers is an F.R.S.-FNRS Research Associate at the ICTEAM Institute,
Universit\'e catholique de Louvain (email: \texttt{raphael.jungers@uclouvain.be}). \rmj{His research is supported by the French Community of Belgium, the Walloon Region, and the Innoviris grant BDL-SMARK.}}}

\begin{document}

\maketitle \thispagestyle{empty} \pagestyle{empty}

%%%%%%%%%%%%%%%%%%%%%%%%%%%%%%%%%%%%%%%%%%%%%%%%%%%%%%%%%%%%%%%%%%%%%%%%%%%%%%%%
\begin{abstract}	
We introduce the concept of \emph{sos-convex} Lyapunov functions for stability analysis of both linear and nonlinear difference inclusions (also known as discrete-time switched systems). These are polynomial Lyapunov functions that have an algebraic certificate of convexity and that can be efficiently found via semidefinite programming. We prove that sos-convex Lyapunov functions are universal (i.e., necessary and sufficient) for stability analysis of switched linear systems. We show via an explicit example however that the minimum degree of a convex polynomial Lyapunov function can be arbitrarily higher than a non-convex polynomial Lyapunov function. In the case of switched \emph{nonlinear} systems, we prove that existence of a common non-convex Lyapunov function does \emph{not} imply stability, but existence of a common convex Lyapunov function does. We then provide a semidefinite programming-based procedure for computing a full-dimensional subset of the region of attraction of equilibrium points of switched polynomial systems, under the condition that their linearization be stable. We conclude by showing that our semidefinite program can be extended to search for Lyapunov functions that are pointwise maxima of sos-convex polynomials.

\end{abstract}


%%%%%%%%%%%%%%%%%%%%%%%%%%%%%%%%%%%%%%%%%%%%%%%%%%%%%%%%%%%%%%%%%%%%%%%%%%%%%%%%

\section{Introduction} The most commonly used Lyapunov functions in control theory, namely the quadratic ones, are convex functions. This convexity property is not always purposefully sought after; it is simply an artifact of the nonnegativity requirement of Lyapunov functions, which for quadratic forms coincides with convexity. If one however seeks Lyapunov functions that are polynomial functions of degree larger than two (for instance, for improving some sort of performance metric), then convexity is no longer implied by the nonnegativity requirement of the Lyapunov function (consider, e.g., the polynomial $x_1^2x_2^2$). In this paper we ask the following question: what is there to gain (or to lose) by requiring that a polynomial Lyapunov function be convex? We also present a computational methodology, based on semidefinite programming, for automatically searching for convex polynomial Lyapunov functions.

% two---a task which can have many merits and has become computationally approachable in the past decade due to advances in sum of squares optimization~\cite{PhD:Parrilo}---

%it is simply an artifact of the equivalence of convexity and nonnegativity for quadratic forms, and the %nonnegativity requirement of a Lyapunov function

Our study of this question is motivated by, and for the purposes of this paper exclusively focused on, the stability problem for difference inclusions, also known as discrete time switched systems. We are concerned with an uncertain and time-varying map
\begin{equation}
x_{k+1}=\tilde{f_k}(x_k), \label{eq:switched.nonlinear.system}
\end{equation}
where 
\begin{equation}
\tilde{f_k}(x_k)\in conv\{f_1(x_k),\ldots,f_m(x_k)\}.\label{eq:ftilda=conv}
\end{equation}
Here, $f_1,\ldots,f_m: \mathbb{R}^n\rightarrow\mathbb{R}^n$ are $m$ different (possibly nonlinear) continuous maps with $f_i(0)=0$, and $conv$ denotes the convex hull operation. The question of interest is (local or global) \emph{asymptotic stability under arbitrary switching}. This means that we would like to know whether the origin is stable in the sense of Lyapunov (see~\cite{Khalil:3rd.Ed} for a definition) and attracts all initial conditions (either in a neighborhood or globally) for \emph{all} possible values that $\tilde{f}_k$ can take at each time step $k$.

The special case of this problem where the maps $f_1,\ldots,f_m$ are \emph{linear} has been and continues to be the subject of intense study in the control community, as well as in the mathematics and computer science communities \cite{jsr-toolbox,morris-ergodic,cfbousch,Shorten05stabilitycriteria,jungers_lncis,LeeD06,bcv12,liberzon-cdc10}. A switched linear system in this setting is given by
\begin{equation}\label{eq:switched.linear.system}
x_{k+1}\in conv\{A_ix_k\}, \quad i=1,\ldots,m, 
\end{equation}
where $A_1,\ldots,A_m$ are $m$ real $n\times n$ matrices. Local (or equivalently global) asymptotic stability under arbitrary switching of this system is equivalent to the \emph{joint spectral radius} of these matrices being strictly less than one. 
\begin{definition}[Joint Spectral Radius (JSR) \cite{RoSt60}] 
% consider
%$\widehat\rho_k(\cM):=\sup_{P\in \P_k(\cM)} \|P\|^{1/k},
%\quad k\in \N$ and 
The {\em joint spectral radius} of a set of matrices $\cM$ is defined as
\begin{equation}\label{eq:JSR}
\rho(\cM)=\lim_{k\rightarrow \infty} \
\max_{A_1,\dots,A_k\in \cM}||A_1\dots A_k||^{1/k},
\end{equation} where 
$\|\cdot\|$ is any matrix norm on $\mathbb{\re}^{n \times n}.$
\end{definition}

Deciding whether $\rho<1$ is notoriously difficult. No finite time procedure for this purpose is known to date, and the related problems of testing whether $\rho \leq 1$ or whether the trajectories of (\ref{eq:switched.linear.system}) are bounded under arbitrary switching are known to be undecidable \cite{BlTi3}. On the positive side however, a large number of sufficient conditions for stability of such systems are known. Most of these conditions are based on the numerical construction of special classes of Lyapunov functions, a subset of which enjoy theoretical guarantees in terms of their quality of approximation of the joint spectral radius \cite{GP11,Ando98,Pablo_Jadbabaie_JSR_journal,protasov-jungers-blondel09,jungersguglielmicicone12}.

It is well known that if the switched linear system (\ref{eq:switched.linear.system}) is stable\footnote{Throughout this paper, by the word ``stable'' we mean asymptotically stable under arbitrary switching.}, then it admits a common \emph{convex} Lyapunov function, in fact a norm \cite{jungers_lncis}. It is also known that stable switched linear systems admit a common \emph{polynomial} Lyapunov function \cite{Pablo_Jadbabaie_JSR_journal}. It is therefore natural to ask whether existence of a common \emph{convex polynomial} Lyapunov function is also necessary for stability. One would in addition want to know how the degree of such convex polynomial Lyapunov function compares with the degree of a non-convex polynomial Lyapunov function. We address both of these questions in this paper. 

It is not difficult to show (see \cite[Proposition 1.8]{jungers_lncis}) that stability of the linear inclusion (\ref{eq:switched.linear.system}) is equivalent to stability of its ``corners''; i.e. to 
stability of a switched system that at each time step applies one of the $m$ matrices $A_1,\ldots,A_m$, but never a matrix strictly inside their convex hull. This statement is no longer true for the switched nonlinear system in (\ref{eq:switched.nonlinear.system})-(\ref{eq:ftilda=conv}); see Example \ref{ex:nonconvex.fails} in Section~\ref{subsec:nl.global} of this paper. It turns out, however, that one can still prove switched stability of the entire convex hull by finding a common \emph{convex} Lyapunov function for the corner systems $f_1,\ldots,f_m$. This is demonstrated in our Proposition \ref{prop:convex.lyap} and Example \ref{ex:convex.lyap}, where we demonstrate that convexity of the Lyapunov function is important in such a setting. 

Such considerations motivate us to seek efficient algorithms that can automatically search over all candidate convex polynomial Lyapunov functions of a given degree.
%
%one would like to have an efficient algorithm that automatically searches over all candidate convex polynomial Lyapunov functions of a given degree. 
This task, however, is unfortunately intractable even when one restricts attention to quartic (i.e., degree-four) Lyapunov functions and switched linear systems. See our discussion in Section~\ref{section:sos-convex}. In order to cope with this issue, we introduce the class of \emph{sos-convex} Lyapunov functions (see Definition \ref{def:sos-convex}). Roughly speaking, these Lyapunov functions constitute a subset of convex polynomial Lyapunov functions whose convexity is certified through an explicit algebraic identity. One can search over sos-convex Lyapunov functions of a given degree by solving a single semidefinite program whose size is polynomial in the description size of the input dynamical system. The methodology can directly handle the linear switched system in (\ref{eq:switched.linear.system}) or its nonlinear counterpart in (\ref{eq:switched.nonlinear.system})-(\ref{eq:ftilda=conv}), if the maps $f_1,\ldots,f_m$ are \emph{polynomial functions}.\footnote{While polynomial dynamical systems are already a broad and significant class of nonlinear dynamical systems, certain extensions are possible. For example, our methodology extends in a straightforward fashion to the case where the functions $f_i$ are rational functions with sign-definite denominators. Extensions to trigonometric dynamical systems may also be possible using the ideas in~\cite{megretski_trig}.} 

We will review some results from the thesis of the first author which show that for certain dimensions and degrees, the set of convex and sos-convex Lyapunov functions coincide. In fact, in relatively low dimensions and degrees, it is quite challenging to find convex polynomials that are not sos-convex \cite{AAA_PP_not_sos_convex_journal}. This is evidence for the strength of this semidefinite relaxation and is encouraging from an application viewpoint. Nevertheless, since sos-convex polynomials are in general a strict subset of the convex ones, a more refined (and perhaps more computationally relevant) converse Lyapunov question for switched linear systems is to see whether their stability guarantees existence of an sos-convex Lyapunov function. This question is also addressed in this paper.

We shall remark that there are other classes of convex Lyapunov functions whose construction is amenable to convex optimization. The main examples include polytopic Lyapunov functions, and piecewise quadratic Lyapunov functions that are a geometric combination of several quadratics \cite{gwz05,protasov2,jungersprotasov09,protasov-jungers-blondel09,BlNes05,dual_LMI_diff_inclusions,LeeD06}. These Lyapunov functions are mostly studied for the case of linear switched systems, where they are known to be necessary and sufficient for stability. The extension of their applicability to polynomial switched systems should be possible via the sum of squares relaxation. Our focus in this paper however is solely on studying the power of sos-convex polynomial Lyapunov functions. Only in our last section, do we briefly comment on extensions to piecewise sos-convex Lyapunov functions.

\subsection{Related work}
The literature on stability of switched systems is too extensive for us to review. We simply refer the interested reader to \cite{Shorten05stabilitycriteria,gst-book,jungers_lncis} and the references therein. Closer to the specific focus of this paper is the work of Mason et al. \cite{mason-boscain-chitour06}, where the authors prove existence of polynomial Lyapunov functions for switched linear systems in \emph{continuous time}. Our proof of the analogous statement in discrete time closely follows theirs. In \cite{ahmadi2017sum}, Ahmadi and Parrilo show that in the continuous time case, existence of the Lyapunov function of Mason et al. further implies existence of a Lyapunov function that can be found with sum of squares techniques. 
%Similar statements are proven there for polynomial differential equations.
In \cite{Pablo_Jadbabaie_JSR_journal}, Parrilo and Jadbabaie prove that stable switched linear systems in discrete time always admit a (not necessarily convex) polynomial Lyapunov function which can be found with sum of squares techniques. Blanchini and Franco show in~\cite{blanchini_no_convex_Lyap} that in contrast to the case of uncontrolled switching (our setting), controlled linear switched systems, both in discrete and continuous time, can be stabilized by means of a suitable switching strategy without necessarily admitting a convex Lyapunov function.

In \cite{Chesi_Hung_journal},~\cite{Chesi_Hung_conf}, Chesi and Hung motivate several interesting applications of working with convex Lyapunov functions or Lyapunov functions with convex sublevel sets. These include establishing more regular behavior of the trajectories, ease of optimization over sublevel sets of the Lyapunov function, stability of recurrent neural networks, etc. The authors in fact propose sum of squares based conditions for imposing convexity of polynomials. However, it is shown in~\cite[Sect. 4]{AAA_PP_CDC10_algeb_convex} that these conditions lead to semidefinite programs of larger size than those of sos-convexity, while at the same time being at least as conservative. Moreover, the works in~\cite{Chesi_Hung_journal},~\cite{Chesi_Hung_conf} do not offer an analysis of the performance (existence) of convex Lyapunov functions.

On the optimization side, the reader interested in knowing more about sos-convex polynomials, their role in convex algebraic geometry and polynomial optimization, and their applications outside of control is referred to the works by Ahmadi and Parrilo~\cite{AAA_PP_not_sos_convex_journal},~\cite{AAA_PP_table_sos-convexity}, Helton and Nie~\cite{Helton_Nie_SDP_repres_2}, and Magnani et al.~\cite{convex_fitting}, or to Section 3.3.3 of the edited volume~\cite{Convex_Alg_Geom_BOOK}. Finally, we note that a shorter version of the current paper with some preliminary results appears in~\cite{aaa_raph_sosconvex_cdc} as a conference paper.

\subsection{Organization and contributions of the paper}
The paper is organized as follows. In Section \ref{section:sos-convex}, we present the mathematical and algorithmic machinery for working with sos-convex Lyapunov functions and explain its connection to semidefinite programming. In Section \ref{section:linear}, we study switched linear systems. We show that given any homogeneous Lyapunov function, the Minkowski norm defined by the convex hull of its sublevel set is also a valid (convex) Lyapunov function (Proposition~\ref{prop:Minkowski}). We then show that any stable switched linear system admits a convex polynomial Lyapunov function (Theorem~\ref{thm:existence-convex-poly}). Furthermore, we give algebraic arguments to strengthen this result and prove existence of an sos-convex Lyapunov function (Theorem~\ref{thm:existence-sos-convex-poly}). While existence of a convex polynomial Lyapunov functions is always guaranteed, we prove that in worst case, the degree of such a Lyapunov function can be arbitrarily higher than that of a non-convex polynomial Lyapunov function (Theorem~\ref{thm:degree.higher}).

%An explicit family of switched linear systems is then provided to show that the minimum degree of a convex polynomial Lyapunov function can in worst case be arbitrarily larger than a non-convex one. 

In Section \ref{section:nonlinear}, we study nonlinear switched systems. We show that stability of these systems cannot be inferred from the existence of a common Lyapunov function for the corner systems (Example~\ref{ex:nonconvex.fails}). However, we prove that this conclusion can be made if the common Lyapunov function is convex (Proposition~\ref{prop:convex.lyap}). We also give a lemma that shows that the radial unboundedness requirement of a Lyapunov function is implied by its convexity (Lemma~\ref{lem:convex.coercive}). We then provide an algorithm based on semidefinite programming that under mild conditions finds a full-dimensional inner approximation to the region of attraction of a locally stable equilibrium point of a polynomial switched system (Theorem~\ref{thm:nl.beta.guarantee}). This algorithm is based on a search for an sos-convex polynomial whose sublevel set is proven to be in the region of attraction via a sum of squares certificate coming from Stengle's Positivstellensatz. Some examples are provided in Section~\ref{subsec:examples}.

Finally, in Section~\ref{sec:future}, we briefly describe some future directions and extensions of our framework to a broader class of convex Lyapunov functions that are constructed from combining several sos-convex polynomials. These extensions are still amenable to semidefinite programming and have connections to the theory of path-complete graph Lyapunov functions proposed in~\cite{ajpr-sicon}.

\section{Sos-convex polynomials}\label{section:sos-convex}
%[To do: AAA - 
%Give three equivalent formulations of sos-convexity. Say how can be searched over by SDP.
%Define what we mean by sos-convex Lyap (decrease conditions need to be sos also). Write a theorem saying that for
%switched linear systems, convex and sos-convex Lyaps are equivalent, if (i) system is planar and degree of Lyap is whatever,
%or (ii) system is in R3 and degree of Lyap is 4.]
%[show that for homogeneous (which we are anyway), convex implies nonnegative and sos-convex implies sos.]

A multivariate polynomial $p(x)\mathrel{\mathop:}=p(x_1,\ldots,x_n)$ is said to be \emph{nonnegative} or \emph{positive semidefinite} (psd) if $p(x)\geq0$ for all $x\in\mathbb{R}^n$. We say that $p$ is a \emph{sum of squares} (sos) if it can be written as  $p=\sum_i q_i^2$, where each $q_i$ is a polynomial. It is well known that if $p$ is of even degree four or larger, then testing nonnegativity is NP-hard, while testing existence of a sum of squares decomposition, which provides a sufficient condition and an algebraic certificate for nonnegativity, can be done by solving a polynomially-sized semidefinite program~\cite{PhD:Parrilo},\cite{sdprelax}.

A polynomial $p\mathrel{\mathop:}=p(x)$ is \emph{convex} if its Hessian $\nabla^2p(x)$ (i.e., the $n\times n$ polynomial matrix of the second derivatives) is \rmj{a} positive semidefinite matrix for all $x\in\mathbb{R}^n$. This is equivalent to the scalar-valued polynomial $y^T\nabla^2p(x)y$ in $2n$ variables $(x_1,\ldots,x_n,y_1,\ldots,y_n)$ being nonnegative. It has been shown in~\cite{NPhard_Convexity} that testing if a polynomial of degree four is convex is NP-hard in the strong sense. This motivates the algebraic notion of \emph{sos-convexity}, which can be checked with semidefinite programming and provides a sufficient condition for convexity.

\begin{definition}\label{def:sos-convex}
A polynomial $p\mathrel{\mathop:}=p(x)$ is \emph{sos-convex} if its Hessian $\nabla^2p(x)$ can be factored as 
$$\nabla^2p(x)=M^T(x)M(x),$$ where $M(x)$ is a polynomial matrix; i.e., a matrix with polynomial entries.
\end{definition}

Polynomial matrices which admit a decomposition as above are called \emph{sos matrices}. The term \emph{sos-convex} was coined in a seminal paper of Helton and Nie~\cite{Helton_Nie_SDP_repres_2}. %, where they prove (among other things) that a basic semialgebraic set defined by sos-convex inequalities always has a lifted semidefinite representation. 
The following theorem is an algebraic analogue of a classical theorem in convex analysis and provides equivalent characterizations of sos-convexity.

\begin{theorem} [Ahmadi and Parrilo~\cite{AAA_PP_table_sos-convexity}] \label{thm:sos.convexity.3.equivalent.defs}
Let $p\mathrel{\mathop:}=p(x)$ be a polynomial of degree $d$ in
$n$ variables with its gradient and Hessian denoted respectively
by $\nabla p\mathrel{\mathop:}=\nabla p(x) $ and
$\nabla^2p\mathrel{\mathop:}=\nabla^2 p(x)$. Let $g_{\lambda}$, $g_\nabla$, and
$g_{\nabla^2}$ be defined as
\begin{equation} \label{eq:defn.g_lambda.g_grad.g_grad2}
\begin{array}{lll}
g_{\lambda}(x,y)&=&(1-\lambda)p(x)+\lambda p(y)-p((1-\lambda)
x+\lambda y),\\
g_\nabla(x,y)&=&p(y)-p(x)-\nabla p(x)^T(y-x), \\
g_{\nabla^2}(x,y)&=&y^{T}\nabla^2p(x)y.
\end{array}
\end{equation}
Then the following are equivalent to sos-convexity of $p$:

\textbf{(a)}  \  $g_{\frac{1}{2}}(x,y)$ is sos\footnote{The
constant $\frac{1}{2}$ in $g_{\frac{1}{2}}(x,y)$ of condition
\textbf{(a)} is arbitrary and chosen for convenience. One can show
that $g_{\frac{1}{2}}$ being sos implies that $g_{\lambda}$ is sos
for any fixed $\lambda\in[0,1]$. Conversely, if $g_{\lambda}$ is
sos for some $\lambda\in(0,1)$, then $g_{\frac{1}{2}}$ is sos.}.

%\textbf{(a')} $g_\lambda(x,y)$ is sos for any fixed
%$\lambda\in[0,1].$

\textbf{(b)}  \  $g_\nabla(x,y)$ is sos.

\textbf{(c)}   \  $g_{\nabla^2}(x,y)$ is sos. %(equivalently $\nabla^2p(x)$ is an
%sos-matrix).
\end{theorem}

The above theorem is reassuring in the sense that it demonstrates the invariance of the definition of sos-convexity with respect to the characterization of convexity that one may choose to apply the sos relaxation to. Since existence of an sos decomposition can be checked via semidefinite programming (SDP), any of the three equivalent conditions above, and hence sos-convexity of a polynomial, can also be checked by SDP. Even though the polynomials $g_{\frac{1}{2}}$, $g_\nabla$, $g_{\nabla^2}$ above are all in $2n$ variables and have degree $d$, the structure of the polynomial $g_{\nabla^2}$ allows for much smaller SDPs (see~\cite{AAA_PP_CDC10_algeb_convex} for details). %Hence, we will use the Hessian condition throughout this paper.

%Examples of convex polynomials that are not sos-convex are known.

In general, finding examples of convex polynomials that are not sos-convex seems to be a nontrivial task, though a number of such constructions are known~\cite{AAA_PP_not_sos_convex_journal}. A complete characterization of the dimensions and the degrees for which the notions of convexity and sos-convexity coincide is available in~\cite{AAA_PP_table_sos-convexity}. %In general, finding examples of convex but not sos-convex polynomials is a challenging task~\cite{AAA_PP_not_sos_convex_journal}. 

Crucial for our purposes is the fact that semidefinite programming allows us to not only check if a given polynomial is sos-convex, but also search and optimize over the set of sos-convex polynomials of a given degree. This feature enables an automated search over a subset of convex polynomial Lyapunov functions. Of course, a Lyapunov function also needs to satisfy other constraints, namely positivity and monotonic decrease along trajectories. Following the standard approach, we replace the  inequalities underlying these constraints with their sum of squares counterparts as well.

% requirement e.g. that $V_k-V_{k+1}$ have a sum of squares decomposition. %This approach which is very standard by now.

Throughout this paper, what we mean by an \emph{sos-convex Lyapunov function} is a polynomial function which is sos-convex and satisfies all other required Lyapunov inequalities with sos certificates.\footnote{Even though an sos decomposition in general merely guarantees \rmj{nonnegativity of a polynomial}, sos decompositions obtained numerically from interior point methods generically provide proofs of \rmj{its} \emph{positivity}; see the discussion in~\cite[p.41]{AAA_MS_Thesis}. In this paper, whenever we are concerned with asymptotic stability and prove a result about existence of a Lyapunov function satisfying certain sos conditions, we make sure that the resulting inequalities are strict (cf. Theorem~\ref{thm:existence-sos-convex-poly}).} When the Lyapunov function can be taken to be homogeneous---as is the case when the dynamics are homogeneous~\cite{HomogHomog}---then the following lemma establishes that the convexity requirement of the polynomial automatically meets its nonnegativity requirement.

Recall that a \emph{homogeneous polynomial} (or a \emph{form}) is a polynomial whose monomials all have the same degree.

\begin{lemma}
Convex forms are nonnegative and sos-convex forms are sos.
\end{lemma}

\begin{proof}
See~\cite[Lemma 8]{Helton_Nie_SDP_repres_2} or~\cite[Lemma 3.2]{AAA_PP_table_sos-convexity}.
\end{proof}

For stability analysis of the switched linear system in (\ref{eq:switched.linear.system}), the requirements of a (common) sos-convex Lyapunov function $V$ are therefore the following:
\begin{equation}\label{eq:sos-convex.Lyap.requirements}
\begin{array}{ll}
V(x) & \mbox{sos-convex}\\
V(x)-V(A_ix) & \mbox{sos for}\ i=1,\ldots,m.
\end{array}
\end{equation}

Given a set of matrices $\{A_1,\ldots,A_m\}$ with rational entries, the search for the coefficients of a fixed-degree polynomial $V$ satisfying the above conditions amounts to solving an SDP whose size is polynomial in the bit size of the matrices. If this SDP is (strictly) feasible, the switched system in (\ref{eq:switched.linear.system}) is stable under arbitrary switching. We remark that the same implication is true if the sos-convexity requirement of $V$ is replaced with the requirement that $V$ is sos; see~\cite[Thm. 2.2]{Pablo_Jadbabaie_JSR_journal}. (This statement fails to hold for switched nonlinear systems.)

In the next section, we study the converse question of existence of a Lyapunov function satisfying the semidefinite conditions in (\ref{eq:sos-convex.Lyap.requirements}).

%As a related remark, we end this section by mentioning that examples of convex polynomials that are not sos-convex are known. In fact, a complete characterization of the dimensions and the degrees for which convexity and sos-convexity are the same notion is available~\cite{AAA_PP_table_sos-convexity},~\cite[Chap. 3]{AAA_PhD}. In general, finding examples of convex but not sos-convex polynomials is a challenging task~\cite{AAA_PP_not_sos_convex_journal}. From an application viewpoint, this is good news. It implies that our sos-convex Lyapunov functions are a powerful replacement for convex polynomial Lyapunov functions.

\section{Sos-convex Lyapunov functions and switched linear systems}\label{section:linear}
%[AAA: show that if a set is invariant, its convex hull is also invariant. Show that if V is a Lyap, its
%convex envelope is also a Lyap (haven't carefully check this yet).]

As remarked in the introduction, it is known that asymptotic stability of a switched linear system under arbitrary switching implies existence of a common convex Lyapunov function, as well as existence of a common polynomial Lyapunov function. In this section, we show that this stability property in fact implies existence of a common Lyapunov function that is both convex and polynomial (cf. Subsection~\ref{subsec:existence.convex}). Moreover, we strengthen this result to show existence of a common sos-convex Lyapunov function (cf. Subsection~\ref{subsec:sos-convex}).

%We furthermore prove that one can take this polynomial Lyapunov function to be \emph{sos-convex}.

%one can in fact conclude existence of a convex polynomial Lyapunov function, and even more, an sos-convex Lyapunov function. 

Before we prove these results, we state a related proposition which shows that in the particular case of switched linear systems, any common Lyapunov function (e.g. a non-convex polynomial) can be turned into a common \emph{convex} Lyapunov function, although not necessarily an efficiently computable one. We believe that this statement must be known, but since we could not pinpoint a reference, we include a proof here.

\begin{proposition}\label{prop:Minkowski}
Consider the switched linear system in (\ref{eq:switched.linear.system}). Suppose $V$ is a common homogeneous and continuous Lyapunov function for (\ref{eq:switched.linear.system}); i.e. satisfies \rmj{$V(x)>0, \forall x\neq 0, $ and $V(x)-V(A_ix)>0, \forall x\neq 0, \forall i=1,\ldots,m.$} Let $$\mathcal{S}\mathrel{\mathop:}=\{x\in\mathbb{R}^n|~\ V(x)\leq 1\}.$$ Then, the Minkowski (a.k.a. gauge) norm defined by the set $\mathcal{S}$, i.e. the function $$W(x)\mathrel{\mathop:}=\inf \{t>0|\ \frac{x}{t}\in conv(\mathcal{S})\},$$
is a convex common Lyapunov function for (\ref{eq:switched.linear.system}).
% with unit sublevel set $\mathcal{S}$, then the Minkowski norm\footnote{The Minkowski (or gauge) norm $q$ defined by an origin-symmetric convex body $\mathcal{S}$ is given by $$q(x)=\inf \{t>0|\ x\in t\mathcal{S}\}.$$} defined by $conv({\mathcal{S})}$ is a convex common Lyapunov function for (\ref{eq:switched.linear.system}).
\end{proposition}
\begin{proof}	Since under the assumptions of the proposition the set $conv(\mathcal{S})$ is compact, origin symmetric, and has nonempty interior, the function $W$ is a norm (see e.g.~\cite[p. 119]{BoydBook}) and hence convexity and positivity of $W$ are already established. It remains to show that for any $i\in\{1,\ldots,m\}$ and $x\neq 0$ we have 	
\begin{equation}\nonumber
\begin{array}{lll}
W(A_ix)&=& \inf \{t>0|\ \frac{A_ix}{t}\in conv(\mathcal{S})\}\\
\ &<& \inf \{t>0|\ \frac{x}{t}\in conv(\mathcal{S})\}\\
\ &=& W(x).
\end{array}
\end{equation}
To see the inequality, first note that because $V$ is a common Lyapunov function, there must exist a constant $\gamma\in(0,1)$ such that if $z\in\mathcal{S}$, then $A_iz\in\gamma \mathcal{S}$. Now observe that if for some $t>0$ we have $\frac{x}{t}\in conv(\mathcal{S}),$ then by definition $\frac{x}{t}=\sum_j \lambda_j y_j$ for some $y_j\in\mathcal{S}$ and $\lambda_j\geq 0$ with $\sum_j \lambda_j =1.$ Hence, $$\frac{A_i x}{t}=\sum_j  \lambda_j A_iy_j=\sum_j  \lambda_j \gamma w_j,$$ for some $w_j\in\mathcal{S}$. But this means that $\frac{A_i x}{\gamma t}\in conv(\mathcal{S}).$
%By homogeneity of linear systems, it suffices to prove that if $\mathcal{S}$ is an invariant set for this dynamics, then so is $conv({\mathcal{S})}$. To see this, take $x_1,x_2 \in S,$ and consider $x=\lambda x_1+\lambda x_3 \in conv()$
\end{proof}

\subsection{Existence of convex polynomial Lyapunov functions}\label{subsec:existence.convex}
The goal of this subsection is to prove the following theorem.

%The goal of this subsection is to prove \emph{necessity} of existence of convex polynomial Lyapunov functions for asymptotically stable switched linear systems.

\begin{theorem}\label{thm:existence-convex-poly}
If the switched linear system (\ref{eq:switched.linear.system}) is asymptotically stable under arbitrary switching, then there exists a convex positive definite homogeneous polynomial $p$ that satisfies $p(A_ix)<p(x)$ for all $x\neq 0$ and all $i\in\{1,\ldots,m\}.$
%	For any asymptotically stable linear switched system, there exists a convex polynomial Lyapunov function, i.e., 
\end{theorem}

Our proof is inspired by \cite{mason-boscain-chitour06}, which proves the existence of a convex polynomial Lyapunov function for continuous time switched systems, but we are not aware of an equivalent statement in discrete time. We \rmj{will} also need the following classical result, which to the best of our knowledge first appears in~\cite{RoSt60}.

% In our proofs, we will need the following classical result, which was first proved in \cite{RoSt60} to the best of our knowledge.
\begin{theorem}[see \cite{RoSt60,jungers_lncis}]\label{thm-rotastrang}
Consider a set of matrices $\cM\subset \re^{n\times n}$ with JSR $\rho$. For all $\epsilon>0$, there exists a vector norm $|\cdot|_\epsilon$ in $\re^n$ such that for any \rmj{matrix} $A$ in $\cM,$
$$|x|_\epsilon \leq 1 \quad \Rightarrow \quad |Ax|_\epsilon \leq \rho + \epsilon.$$
%where $\rho$ is the JSR of $\cM$.
\end{theorem}

%\comrj{in fact pablo and ali might do that? if yes we should remove these phrases. if you're not sure I'll check}  Moreover, we %emphasize the convexity of the polynomial function we obtain. 

\begin{proof}(of Theorem \ref{thm:existence-convex-poly}.) Let $\cM\mathrel{\mathop:}=\{A_1,\ldots,A_m\}$ and denote the JSR of $\cM$ by $\rho$. By assumption we have $\rho<1$ and by Theorem \ref{thm-rotastrang}, there exists a norm, which from here on we simply denote by $|\cdot|$, such that $\forall i\in\{1,\ldots,m\},$
$$|x| \leq 1  \Rightarrow \quad |A_ix| \leq \rho + \frac12 (1-\rho). $$ 
We denote the unit ball of this norm by $B$ and use the notation $$\cM B\mathrel{\mathop:}=\{A_ix|~A_i\in\cM \mbox{ and } x\in B\}.$$ Hence, we have $\cM B \subseteq (\rho + \frac12 (1-\rho)) B . $ %(We use the notation $\cM B=\{Ax:A\in\cM \mbox{ and } x\in B\}.$) 

The goal is to construct a convex positive definite homogeneous polynomial $p_d$ of some degree $2d$, such that its 1-sublevel set $S_d$ satisfies $$(\rho + \frac12(1-\rho)) B \subseteq int(S_d) \subset S_d \subseteq B.$$ As $S_d \subseteq B$ and $\cM B \subseteq (\rho + \frac12 (1-\rho)) B$, it would follow that $$\cM S_d \subseteq \cM B \subseteq int(S_d).$$ This would imply that $p_d(A_ix)<p(x),\forall x \in \partial S_d,$ and for $i=1,\ldots,m.$ By homogeneity of $p_d$, we get the claim in the statement of the theorem.

To construct $p_d$, we proceed in the following way. Let 
$$C=\{x\in \mathbb{R}^n |~|x|\leq \rho + \frac{3}{4}(1-\rho)\}.$$ To any $x\in \partial C,$ we associate a (nonzero) dual vector $v(x)$ orthogonal to a supporting hyperplane of $C$ at $x$. This means that $\forall y \in C,\ v(x)^Ty\leq v(x)^Tx$. Since $x\in \inter{B},$ the set $$ S(x)=\{y\in \mathbb{R}^n|~v(x)^Ty>v(x)^Tx \mbox{ and }|y|=1\}$$ is a relatively open nonempty subset of the boundary $\partial B$ of our unit ball. Moreover, $\frac{x}{|x|}\in S(x).$ Now, the family of sets $S(x)$ is an open covering of  $\partial B,$ and hence we can extract a set of points $x_1,\dots,x_N$ such that the union of the sets $S(x_i)$ covers  $\partial B.$ Let $v_i\mathrel{\mathop{:}}=v(x_i).$ For any natural number $d$, we define \footnote{Note that $v_i^Tx_i \neq 0$. In fact, we have $v_i^Tx_i>0$, $\forall i$. Indeed, there exists $\alpha_i>0$ such that $\alpha_iv_i \in C$ and hence $v_i^Tx_i \geq \alpha_i ||v_i||_2^2>0.$} $$p_d(y)=\sum_{i=1}^N \left(\frac{v_i^Ty}{v_i^Tx_i}\right)^{2d} \text{ and } S_d=\{y \in \mathbb{R}^n|~ p_d(y)\leq 1\}.$$
Note that $p_d$ is convex as the sum of even powers of linear forms and homogeneous.
We first show that $$(\rho + \frac12(1-\rho)) B \subseteq int(S_d).$$ As $(\rho + \frac12(1-\rho)) B \subset int(C)$, for all $y \in (\rho + \frac12(1-\rho)) B$ and for all $i=1,\ldots,N$, we have $v_i^Ty < v_i^Tx_i$. Hence there exists a positive integer $d$ such that $$\left(\max_i \max_{y \in (\rho +\frac12 (1-\rho))B} \frac{v_i^Ty}{v_i^Tx_i}\right)^{2d}<\frac{1}{N}$$
and so $p_d(y)<1$ for all $y \in (\rho + \frac12(1-\rho)) B$.

We now show that $S_d \subseteq B.$ Let $y \in S_d$, and so $p_d(y)\leq 1.$ This implies that $$\frac{v_i^Ty}{v_i^Tx_i} \leq 1, \forall i=1,\ldots,N.$$ From this, we deduce that $y \notin \partial B.$ Indeed if $y \in \partial B,$ there exists $i \in \{1,\ldots,N\}$ such that $y \in S(x_i)$, which implies that $v_i^T y>v_i^Tx_i$ and contradicts the previous statement. Hence $\partial B \cap S_d =\emptyset$. As both $B$ and $S_d$ contain the zero vector, we conclude that $S_d \subseteq int(B) \subseteq B$. Note that this guarantees positive definiteness of $p_d$ as $p_d$ is homogeneous and its 1-sublevel set is bounded. \end{proof}

\subsection{Existence of sos-convex polynomial Lyapunov functions} \label{subsec:sos-convex}
%The previous subsection established existence of a convex polynomial Lyapunov function for any asymptotically stable switched linear system. However, we are interested in proving a stronger result which states that even if we were to require the additional algebraic requirement that sos-convex polynomial Lyapunov functions have to meet, one is still guaranteed to

We now strengthen the converse result of the previous subsection by showing that asymptotically stable switched linear systems admit an \emph{sos-convex} Lyapunov function. This in particular implies that such a Lyapunov function can be found with semidefinite programming.

Recall that a homogeneous polynomial $h$ is said to be \emph{positive definite (pd)} if $h(x)>0$ for all $x\neq 0$.

\begin{theorem}\label{thm:existence-sos-convex-poly}
If the switched linear system (\ref{eq:switched.linear.system}) is asymptotically stable under arbitrary switching, then there exists a homogeneous polynomial $q$ that satisfies the sum of squares constraints
\begin{equation}\nonumber
\begin{array}{ll}
q(x) & \mbox{sos-convex},\\
q(x)-q(A_jx) & \mbox{sos for}\ j=1,\ldots,m.
\end{array}
\end{equation}
Moreover, this polynomial $q$ is positive definite and is such that the $m$ polynomials $q(x)-q(A_jx)$ are also positive definite.
\end{theorem}

%We will give the main idea of the proof but the details are omitted and will appear in a journal version of this work.

Our proof will make crucial use of the following Positivstellensatz due to Scheiderer. 

\begin{theorem}[Scheiderer~\cite{Claus_Hilbert17}] \label{thm:claus}
Given any two positive definite homogeneous polynomials $h$ and
$g$, there exists a positive integer $N$ such that $hg^N$ is sos.
\end{theorem}

\begin{proof}(of Theorem \ref{thm:existence-sos-convex-poly}). We have already shown in the proof of Theorem~\ref{thm:existence-convex-poly} that under our assumptions, there exist vectors $a_1,\ldots,a_N\in\mathbb{R}^n$ and a positive integer $d$ such that the convex form $p(x)=\sum_{i=1}^N (a_i^Tx)^{2d}$ is positive definite and makes the $m$ forms $p(x)-p(A_jx)$ also positive definite. Note also that as a sum of powers of linear forms, $p$ is already sos-convex and sos. %The goal is now to show that there exists a positive definite form $q$ that is sos such that $q(x)-q(A_jx)$ is both positive definite and sos for any $j=1,\ldots,m$. 
Let $S^{n-1}$ denote the unit sphere in $\mathbb{R}^n$ and define $$\alpha_j\mathrel{\mathop{:}}=\frac{1}{2} \min_{x\in S^{n-1}} p(x)-p(A_jx).$$ By definition of $\alpha_j$, we have that \rmj{$p(x)-p(A_jx)-\alpha_j (\sum_i x_i^2)^d$} is positive definite. Furthermore, as $\alpha_j>0$ (this is a consequence of $p(x)-p(A_jx)$ being positive definite) and as $p(A_jx)$ is sos, we get that $p(A_jx)+\alpha_j (\sum_i x_i^2)^d$ is positive definite. \rmj{Hence, from Theorem~\ref{thm:claus},} there exist an integer $K$ such that 
	\begin{align} \label{eq:claus.p}
		\left(p(x)-p(A_jx)-\alpha_j (\sum_i x_i^2)^d\right)\cdot p(x)^K
	\end{align}
 is sos and an integer $K'$ such that 
 \begin{align}\label{eq:claus.pAj}
 \left(p(x)-p(A_jx)-\alpha_j (\sum_i x_i^2)^d\right) \cdot \left(p(A_jx)+\alpha_j (\sum_i x_i^2)^d\right)^{K'}
 \end{align}  is sos.
	
Take $k\mathrel{\mathop{:}}=\max\{2K+1,2K'+1\}$ and define $q(x)=p(x)^k.$ It is easy to see that $q$ is positive definite as $p$ is positive definite. We first show that $q$ is sos-convex. We have 
$$\nabla^2 q(x)=k(k-1)p(x)^{k-2}\nabla p(x)\nabla p(x)^T+kp(x)^{k-1}\nabla^2 p(x).$$
As $p$ is sos, any power of it is also sos. Furthermore, we have $$\nabla^2 p(x)=\sum_{i=1}^N 2d(2d-1)(a_i^Tx)^{2d-2}a_ia_i^T,$$ which implies that there exists a polynomial matrix $V(x)$ such that $\nabla^2p(x)=V(x)V(x)^T$. As a consequence, we see that 
$$y^T\nabla^2 q(x)y=k(k-1)p(x)^{k-2} (\nabla p(x)^Ty)^2+kp(x)^{k-1}(V(x)^Ty)^2$$ is a sum of squares and hence $q$ is sos-convex.
	
	We now show that for $j=1,\ldots,m$, the form $q(x)-q(A_jx)$ is positive definite and sos. For positive definiteness, simply note that as $p(x)>p(A_jx)$ for any $x \neq 0$ and $p$ is nonnegative, we get $p^k(x)>p^k(A_jx)$ for any $k\geq 1$ and $x\neq 0.$ % Together with the fact that $q(x)-q(A_jx)=0$ when $x=0$, this implies that $q(x)-q(A_jx)$ is positive definite.
	
	To show that $q(x)-q(A_jx)$ is sos, we make use of the following identity:
	\begin{align}\label{eq:identity}
	a^k-b^k=(a-b)\sum_{l=0}^{k-1} a^{k-1-l}b^l.
	\end{align}
Applying (\ref{eq:identity}), we have 
	\begin{align}
	&p^k(x)-\left(p(A_jx)+\alpha_j (\sum_i x_i^2)^d\right)^k \nonumber \\ &=\left(p(x)-p(A_jx)-\alpha_j (\sum_i x_i^2)^d\right)\cdot \sum_{l=0}^{k-1}p(x)^{k-1-l}\left(p(A_jx)+\alpha_j (\sum_i x_i^2)^d\right)^l\nonumber\\
	&=\sum_{l=0}^{k-1}\left(p(x)-p(A_jx)-\alpha_j (\sum_i x_i^2)^d\right)\cdot p(x)^{k-1-l}\left(p(A_jx)+\alpha_j (\sum_i x_i^2)^d\right)^l. \label{eq:sum}
	\end{align}
	For any $l \in \{0,\ldots,k-1\}$, either $k-1-l\geq \frac{k-1}{2}$ or $l \geq \frac{k-1}{2}$. Suppose that the index $l$ is such that $k-1-l \geq \frac{k-1}{2}$: by definition of $k$, this implies that $k-1-l\geq K$. Since the polynomial in (\ref{eq:claus.p}) sos and since $(p(A_jx)+\alpha_j (\sum_i x_i^2)^d)^l$ is sos, we get that the term 	
	$$\left(p(x)-p(A_jx)-\alpha_j (\sum_i x_i^2)^d\right)\cdot p(x)^{k-1-l}\left(p(A_jx)+\alpha_j (\sum_i x_i^2)^d\right)^l$$
	in the sum (\ref{eq:sum}) is sos. Similarly, if the index $l$ is such that $l \geq \frac{k-1}{2}$, we have that $l \geq K'$ by definition of $k.$ Sine the polynomial in (\ref{eq:claus.pAj}) is sos, we come to the conclusion that the term
	$$\left(p(x)-p(A_jx)-\alpha_j (\sum_i x_i^2)^d\right)\cdot p(x)^{k-1-l}\left(p(A_jx)+\alpha_j (\sum_i x_i^2)^d\right)^l$$
	in the sum (\ref{eq:sum}) is sos. When summing over all possible $l \in \{0,\ldots,k-1\}$, as each term in the sum is sos, we conclude that the sum $$ p^k(x)-\left(p(A_jx)+\alpha_j (\sum_i x_i^2)^d\right)^k$$ itself is sos. Now note that we can write
	\begin{align*}
	&p^k(x)-\left(p(A_jx)+\alpha_j (\sum_i x_i^2)^d\right)^k=p(x)^k-p(A_j)^k -\sum_{s=1}^k \binom{k}{s} \alpha_j^s (\sum_i x_i^2)^{ds}\cdot p(A_jx)^{k-s},
	\end{align*}
	which enables us to conclude that $$q(x)-q(A_jx)=p^k(x)-\left(p(A_jx)+\alpha_j (\sum_i x_i^2)^d\right)^k+\sum_{s=1}^k \binom{k}{s} \alpha_j^s (\sum_i x_i^2)^{ds}\cdot p(A_jx)^{k-s}$$
	is sos as the sum of sos polynomials.
\end{proof}

%The strategy of the proof is to start with the convex polynomial Lyapunov function $p_d$ (for a large enough fixed $d$) constructed in the previous subsection and turn it into an sos-convex Lyapunov function $q$. It turns out that we can take $$q(x)=p_d^k(x),$$ for some big enough integer $k$. Here, $p_d$ is the polynomial given in (\ref{eq:construction.of.convex.poly}). It is easy to see that linear forms are sos-convex, and that sums and even powers of sos-convex forms are sos-convex. Therefore, the polynomial $q$ constructed this way is sos-convex. To show that the polynomials $$q(x)-q(A_ix)=p_d^k(x)-p_d^k(A_ix)$$ are all sos, one uses the algebraic identity $$a^k-b^k=(a-b)\sum_{l=0}^{k-1} a^{k-1-l}b^l$$ and appropriately applies Theorem~\ref{thm:claus}.
%
%Finally, we remark that because of the way $q$ is constructed, all of our sos conditions imply strict positivity. So this polynomial will be a strictly feasible solution to a large enough semidefinite program.

\subsection{Non-existence of a uniform bound on the degree of convex polynomial Lyapunov functions}

It is known that there are families of $n\times n$ matrices $\cM=\{A_1,\ldots,A_m\}$ for which the switched linear system (\ref{eq:switched.linear.system}) is asymptotically stable under arbitrary switching, but such that the minimum degree of a common polynomial Lyapunov function is arbitrarily large~\cite{aaa_raph_lower_bounds}. (In fact, this is the case already when $m=n=2$.) In the case where $\cM$ admits a common polynomial Lyapunov function of degree $d$, it is natural to ask whether one can expect $\cM$ to also admit a common \emph{convex} polynomial Lyapunov function of some degree $\hat{d}$, where $\hat{d}$ is a function of $d,n,m$ only? In this subsection, we answer this question in the negative.

%
%
%Theorem \ref{thm:existence-convex-poly} tells us in fact that one can approximate with an arbitrary accuracy the JSR of a set of matrices by restricting the family of polynomial Lyapunov functions to convex polynomials.  
%%We now prove that this result cannot be strengthened in the sense that one cannot ensure exact computation with this restriction.  More precisely, we show that if the JSR is exactly one, then the existence of a (nonstrict) Convex polynomial Lyapunov function is not guaranteed, even if there exists a (nonconvex) polynomial Lyapunov function. 
%Note that it is known that there are stable sets of matrices with polynomial Lyapunov functions of arbitrary degree, but one could wonder whether the existence of a {polynomial Lyapunov function} of a certain degree actually implies a bound on the degree of a \emph{convex} Lyapunov function.
%We show that this is not true in the next example.

Consider the set of matrices $\mathcal{A}=\{A_1,A_2\},$ with
\begin{equation}\label{eq:A1,A2.ando.shi}
A_{1}=\left[
\begin{array}
[c]{cc}%
1 & 0\\
1 & 0
\end{array}
\right]  ,\text{ }A_{2}=\left[
\begin{array}
[c]{cr}%
0 & 1\\
0 & -1
\end{array}
\right].
\end{equation}
This is a benchmark set of matrices that has been studied
in~\cite{Ando98},~\cite{Pablo_Jadbabaie_JSR_journal}
mainly because it provides a ``worst-case'' example for the method of common
quadratic Lyapunov functions. Indeed, it is easy to show that
$\rho(\mathcal{A})=1$, but a common quadratic Lyapunov function can only produce an upper bound of $\sqrt{2}$ on the JSR. In~\cite{Pablo_Jadbabaie_JSR_journal}, Parrilo and Jadbabaie give a simple degree-4 (non-convex) common polynomial Lyapunov function that proves stability of the switched linear system defined by the matices $\{\gamma A_1, \gamma A_2\},$ for any $\gamma<1$. In sharp contrast, we show the following:

%$\hat{\rho}_{\mathcal{V}^2}(\mathcal{A})=\sqrt{2}$. Moreover, the
%bound obtained by a common quadratic function applied to the set
%$\mathcal{A}^t$ is
%$$\hat{\rho}_{\mathcal{V}^2}^{\frac{1}{t}}(\mathcal{A}^t)=2^{\frac{1}{2t}},$$
%which for no finite value of $t$ is exact.

%\comrj{I should mention 'homogeneous' in the next thm, right?  }
%AAA: no because non-homogeneity doesn't buy us anything.
%yes bet in the proof I assume the homogeneity of the convex poly.  If it were not homogeneous,

\begin{theorem}\label{thm:degree.higher}
Let $A_1,A_2$ be as in (\ref{eq:A1,A2.ando.shi}) and consider the sets of matrices $\cM_\gamma=\{\gamma A_1, \gamma A_2\}$ parameterized by a scalar $\gamma <1$. As $\gamma\rightarrow 1$, the minimum degree of a common convex polynomial Lyapunov function for $\cM_\gamma$ goes to infinity. 
\end{theorem}

\begin{proof}
It is sufficient to prove that the set $\{A_1,A_2\}$ has no convex invariant set defined as the sublevel set of a polynomial.  Indeed, if there were a uniform bound $D$ on the degree of a convex polynomial Lyapunov function \rmj{for the sets $\{\gamma A_1,\gamma A_2\},$} this would imply the existence of an invariant set---which is the sublevel set of a convex polynomial function of degree $D$---\rmj{for the set $\{A_1,A_2\}$ itself.}\\
We prove our claim by contradiction.  In fact, we will prove the slightly stronger fact that for these matrices, the only convex invariant set is the unit square $$S=\{(x,y)\in\mathbb{R}^2|\ ||(x,y)||_\infty\leq 1\},$$ or, of course, a scaling of it. 

Let $\mathcal{A}=\{A_1,A_2\}$ and let $\mathcal{A}^*$ denote the set of all matrix products out of $\mathcal{A}$. Suppose for the sake of contradiction that there was a convex bivariate polynomial $p$ whose unit level set was the boundary of an invariant set for the switched system defined by $\mathcal{A}$.  More precisely, suppose we had \begin{equation}\label{eqn-contrex-lyap}\forall x \in \re^2,\ \forall A\in \mathcal{A}, \quad p(Ax) \leq p(x).\end{equation}
Let $x^*\in\mathbb{R}$ be such that \rmj{$$p(x^*,x^*)=1.$$}
It is easy to check that the following matrices can be obtained as products of matrices in $\mathcal{A}$: 
\begin{equation}\label{eqn-ando-semigroup}
\left \{  
\begin{pmatrix}
0 & 1\\
0 & -1
\end{pmatrix},\ \begin{pmatrix}
0 & 1\\
0 & 1
\end{pmatrix}
         ,\
\begin{pmatrix}
0 & -1\\
0 & 1
\end{pmatrix},\right \}\quad \subset \cA^*.
\end{equation}
This implies that \begin{eqnarray}\nonumber p({\bf x})&=&1 \\
\nonumber \mbox{for } {\bf x}&\in&\{(x^*,-x^*),(-x^*,-x^*),(-x^*,x^*)\}\end{eqnarray} as well, because these points can all be mapped onto each other with matrices from (\ref{eqn-ando-semigroup}). 

Suppose that there is an $x>x^*,$ $-x^*<y<x^*,$ such that \rmj{$p(x,y)= 1.$} Then we reach a contradiction because (\ref{eqn-ando-semigroup}) implies that $(x,y)$ can be mapped on $(x,x),$ which contradicts (\ref{eqn-contrex-lyap}) because $x>x^*.$ This implies that $p(x^*,y)\geq 1, \forall y\in (-x^*,x^*)$. However, convexity of $p$ implies that $p(x^*,y)\leq 1, \forall y\in (-x^*,x^*)$. Thus, we have proved that $p(x^*,y)=1,  \forall y\in (-x^*,x^*).$ The same is true for $p(-x^*,y)$ by symmetry.

In the same vein, if there is a $y>x^*,$ $-x^*<x<x^*$ such that $p(x,y)=1,$ this point can be mapped on $(-y,-y),$ which again leads to a contradiction, because $p(-x^*,-x^*)=1.$ Hence,  $p(x,x^*)=p(x,-x^*)=1, \forall x\in(-x^*,x^*),$ which concludes the proof.
\end{proof}
%[AAA: Plot a picture of the quartic Lyap of Pablo and Jadbabaie to show its convex hull is the square. Forget it - we leave for journal version.]
%
%[AAA: say that since Ando and Shih is on the plane, convex=sos-convex. Therefore we can use SDP to find the exact
%upper bound on JSR that a convex Lyap can give for every degree. Say for degree 2,4,6 we get roots of 2, but for degree 8
%we do better - strange. Forget it - we leave it for journal version.]

\section{SOS-convex Lyapunov functions and switched nonlinear systems}\label{section:nonlinear}
%In this section, we demonstrate a noteworthy application of the computational machinery of sos-convex Lyapunov functions, namely the stability analysis of switched nonlinear systems.  These are the systems satisfying these equations:

In this section, we turn our attention to stability analysis of switched nonlinear systems

\begin{eqnarray}\label{eq:nl.systems.again}
x_{k+1}&=&\tilde{f}(x_k), \label{eq:switched.nonlinear.system.ex}\\ \nonumber
\tilde{f}(x_k)&\in & conv\{f_1(x_k),\dots, f_m(x_k)\},
\end{eqnarray}
where $f_1,\ldots,f_m: \mathbb{R}^n\rightarrow\mathbb{R}^n$ are continuous and satisfy $f_i(0)=0.$ We start by demonstrating the significance of convexity of Lyapunov functions in this setting. We then consider the case where $f_1,\ldots,f_m$ are polynomials and devise algorithms that under mild conditions find algebraic certificates of local asymptotic stability under arbitrary switching. These algorithms produce a full-dimensional inner approximation to the region of attraction of the origin, which comes in the form of a sublevel set of an sos-convex polynomial.

%can find algebraic certificates of local asymptotic stability of such systems under arbitrary switching in the case where the functions $f_i$ are polynomial.

\subsection{The significance of convexity of the Lyapunov function}\label{subsec:nl.global}
The following example demonstrates that unlike the case of switched linear systems, one \emph{cannot} simply resort to a common Lyapunov function for the maps $f_1,\ldots, f_m$ to infer a proof of stability of a nonlinear difference inclusion.

% We start by showing on an example the significance of \emph{convex} Lyapunov functions.
\begin{example}\label{ex:nonconvex.fails}
Consider the nonlinear switched system (\ref{eq:switched.nonlinear.system.ex}) with $m=n=2$ and
\begin{eqnarray}
f_1(x)&=&(x_1x_2,0)^T,\label{eq:ex.nonconvex.fails}\\\nonumber
f_2(x)&=&(0,x_1x_2)^T.\end{eqnarray}
\rmj{The} function \begin{equation}\label{eq:Lyap.nonconvex}V(x)=x_1^2x_2^2+(x_1^2+x_2^2)
\end{equation} is a common Lyapunov function for both $f_1$ and $f_2$, but nevertheless the system in (\ref{eq:switched.nonlinear.system.ex}) is unstable.

To see this, note that $$V(f_i(x))=x_1^2x_2^2<V(x)=x_1^2x_2^2+(x_1^2+x_2^2) $$ for $i=1,2,$ and for all $x\neq 0.$\\ On the other hand, (\ref{eq:switched.nonlinear.system.ex}) is unstable since in particular the dynamics $x_{k+1}=f(x_k)$ with $$f(x)=\left(\frac{x_1x_2}{2},\frac{x_1x_2}{2}\right)\in  conv\{f_1(x),f_2(x)\}$$ is obviously unstable.
\end{example}

%Thus, unlike for linear switching systems, one cannot resort to plain Lyapunov functions of the \lq corners\rq{} to prove the stability of a nonlinear switching system (or even to prove their \emph{robust stability}). 

Note that the Lyapunov function in (\ref{eq:Lyap.nonconvex}) was not convex. Proposition~\ref{prop:convex.lyap} below shows that a convexity requirement on the Lyapunov function gets around the problem that arose above. To prove this proposition, we first give a lemma which is potentially of independent interest for global stability analysis. Recall that Lyapunov's theorem for global asymptotic stability commonly requires that the Lyapunov function $V$ be radial unbounded (i.e., satisfy $||x||\rightarrow\infty\implies V(x)\rightarrow\infty$). Our lemma shows that convexity brings this property for free.\footnote{We remind the reader that radial unboundedness is not equivalent to radial unboundedness along restrictions to all lines, hence the need for the subtleties in this proof.}

%Recall that a function $V:\mathbb{R}^n\rightarrow\mathbb{R}$ is \emph{radially unbounded} if $||x||\rightarrow\infty\implies V(x)\rightarrow\infty$. We also note that convex polynomials of degree larger than one are automatically radially unbounded *.
% and let us use the notation $\Delta_m$ to refer to the unit simplex in $\mathbb{R}^n.$

%However, we show now that \emph{convex} Lyapunov functions are indeed a sufficient condition for switched stability.

\begin{lemma}\label{lem:convex.coercive}
Suppose a function $V:\mathbb{R}^n\rightarrow\mathbb{R}$ satisfies $V(0)=0$ and $V(x)>0$ for all $x\neq 0$. If $V$ is convex, then it is radially unbounded.
\end{lemma}

\begin{proof}
	We proceed by contradiction. Suppose that $V$ is not radially unbounded. This implies that there exists a scalar $s>0$ for which the sublevel set $$S\mathrel{\mathop:}=\{x\in\mathbb{R}^n|\ V(x)\leq s  \}$$
	of $V$ is unbounded. As $V$ is convex, $S$ is convex, and as any nonempty sublevel set of $V$ contains the origin, $S$ contains the origin. We claim that $S$ must in fact contain an entire ray originating from the origin.
	
Indeed, as $S$ is unbounded, there exists a sequence of points $\{x_k\}$ such that $\lim_{k\rightarrow \infty}||x_k||=\infty$ and such that $V(x_k)\leq s$ for all $k \in \mathbb{N}.$  Consider now the sequence $\{x_k/||x_k||\}$: this is a bounded sequence and hence has a subsequence that converges. Let $\hat{x}$ be the limit of this subsequence. We argue that the ray $\{c\hat{x}|\ c\geq 0 \}$ is contained in $S$. Suppose that it was not: then $V(\alpha \hat{x})>s$ for some fixed $\alpha>0$, and since $S$ is closed (as a sublevel set of a continuous function), there exists a scalar $\epsilon>0$ such that for all $y\in\mathbb{R}^n$ with $||y-\alpha\hat{x}||\leq \epsilon$, we have $V(y)>s.$ As $\lim_{k\rightarrow \infty} ||x_k||=\infty$ and a subsequence of $\{x_k/||x_k||\}$ converges to $\hat{x}$, there must exist an integer $k_0$ such that $$||x_{k_0}||> \alpha \text{ and  } ||\hat{x}-\frac{x_{k_0}}{||x_{k_0}||}||\leq \epsilon/ \alpha.$$ Note that $$||\alpha \hat{x} -\alpha \frac{x_{k_0}}{||x_{k_0}||}|| \leq \epsilon,$$ which implies that $V(\alpha x_{n_0}/||x_{k_0}||)>s$ and hence $\alpha x_{k_0}/||x_{k_0}||$ does not belong to $S$. But this contradicts convexity of $S$ as $$\frac{\alpha x_{k_0}}{||x_{k_0}||}=\frac{\alpha}{||x_{k_0}||} \cdot x_{k_0}+(1-\frac{\alpha}{||x_{k_0}||})\cdot 0$$
	and $x_{k_0}$ and $0$ are in $S$.
	
We now consider the restriction of $V$ to this ray, which we denote by $g(z)=V(z\hat{x})$, where $z\geq 0.$ We remark that as a univariate function, $g$ is convex, and positive everywhere except at zero where it is equal to zero. By convexity of $g$, we have the inequality $$\frac{1}{w} g(w)+\left(1-\frac{1}{w}\right)g(0)\geq g \left( \frac{1}{w}\cdot w+(1-1/w) \cdot 0 \right)$$ for all $w \in \mathbb{N}$. This is equivalent to \begin{equation}\label{eq:g.univariate.inequality}
\frac{g(w)}{w} \geq g(1).
\end{equation}
Note that $g(1)>0$, but $g(w)\leq s~\forall w \in \mathbb{N}$ since $g$ is a restriction of $V$ to a ray contained in $S$. This contradicts the inequality in (\ref{eq:g.univariate.inequality}) when $w$ is large. Hence, $S$ cannot be unbounded and it follows that $V$ must be radially unbounded.
\end{proof}

\begin{proposition}\label{prop:convex.lyap}
Consider the nonlinear switched system in (\ref{eq:switched.nonlinear.system.ex}).
\begin{enumerate}[(i)]
\item  If there exists a convex function $V:\mathbb{R}^n\rightarrow\mathbb{R}$ that satisfies $V(0)=0$, $V(x)>0$ for all $x\neq 0,$ and 
\begin{equation}\label{eq:corner.decrease}
V(f_i(x))<V(x), \quad \forall x\neq 0, \forall i\in\{1,\ldots,m\},
\end{equation}
then the origin is globally asymptotically stable under arbitrary switching.

\item If there exist a scalar $\beta>0$ and a convex function $V:\mathbb{R}^n\rightarrow\mathbb{R}$ that satisfies $V(0)=0$, $V(x)>0$ for all $x\neq 0,$ and 
\begin{equation}\label{eq:corner.decrease.local}
V(f_i(x))<V(x), \quad \forall x\neq 0 \ \mbox{with}\  V(x)\leq \beta, \mbox{and}\  \forall i\in\{1,\ldots,m\},
\end{equation}
then the origin is locally asymptotically stable under arbitrary switching and the set $\{x\in \mathbb{R}^n|~V(x)\leq \beta \}$ is a subset of the region of attraction of the origin.
%\item If there exist a scalar $\epsilon>0$ and a convex function $V:\mathbb{R}^n\rightarrow\mathbb{R}$ that satisfies $V(0)=0$, $V(x)>0$ for all $x\neq 0,$ and 
%\begin{equation}\label{eq:corner.decrease.local}
%V(f_i(x))<V(x), \quad \forall x\neq 0 \ \mbox{with}\  ||x||\leq \epsilon, \mbox{and}\  \forall i\in\{1,\ldots,m\},
%\end{equation}
%then the origin is locally asymptotically stable under arbitrary switching. Moreover, the largest sublevel set of $V$ which is contained within the ball of radius $\epsilon$ centered at the origin is a subset of the region of attraction.
\end{enumerate}

\end{proposition}  
\begin{proof}
The proof of this proposition is similar to the standard proofs of Lyapunov's theorem except for the parts where convexity intervenes. Hence we only prove part (i) and leave the very analogous proof of part (ii) to the reader.

Suppose the assumptions of (i) hold. Then, for all $x_k \neq 0$ we have
\begin{equation}\label{eq:Lyap.decrease}
\begin{aligned} V(x_{k+1})-V(x_k)&=V\left(\sum_{i=1}^{m}{\lambda_i(k) f_i(x_k)}\right)-V(x_k) \\ &\leq \sum_{i=1}^{m}{\lambda_i(k) V(f_i(x_k))}-V(x_k)\\ &=  \sum_{i=1}^{m}{\lambda_i(k) \left(V(f_i(x_k))-V(x_k)\right)}\\ &<0,
\end{aligned}
\end{equation}
where the first inequality follows from convexity of $V,$ and the second from (\ref{eq:corner.decrease}) and the fact that $\sum_{i=1}^{m}\lambda_i(k)=1$. Hence, our Lyapunov function decreases in each iteration independent of the realization of the uncertain and time-varying map $\tilde{f}$ in (\ref{eq:switched.nonlinear.system.ex}).

To show that the origin is stable in the sense of Lyapunov, consider an arbitrary scalar $\delta>0$ and the ball $B(0,\delta)\mathrel{\mathop:}=\{x\in\mathbb{R}^n|\  ||x||\leq \delta  \}.$ Recall that as a consequence of Lemma~\ref{lem:convex.coercive}, all sublevel sets of $V$ are bounded. Let $\hat{\delta}>0$ be the radius of a ball that is contained in a (full-dimensional) sublevel set of $V$ which itself is contained in $B(0,\delta)$. Then, from (\ref{eq:Lyap.decrease}), we get that $x_0\in B(0,\hat{\delta})\implies x_k\in B(0,\delta), \forall k.$

To show that the origin attracts all initial conditions, consider an arbitrary nonzero point $x_0\in\mathbb{R}^n$ and denote by $\{x_k\}$ any sequence that this initial condition can generate under the iterations of (\ref{eq:switched.nonlinear.system.ex}). We know \rmj{that} the sequence $\{V(x_k)\}$ is positive and decreasing (unless $x_k$ in finite time lands on the origin, in which case the proof is finished). It follows that $\{V(x_k)\}\rightarrow c$ for some scalar $c\geq 0.$ We claim that $c=0$, in which case we must have $x_k\rightarrow 0$ as $k\rightarrow\infty$ which is the desired statement. Suppose for the sake of contradiction that we had $c>0.$ Then, we must have $$x_k\in\Omega\mathrel{\mathop:}=\{x\in\mathbb{R}^n| \ c\leq V(x)\leq V(x_0)   \}, \forall k.$$

Note that the set $\Omega$ is closed and bounded as $V$, being a convex function, is continuous and by Lemma~\ref{lem:convex.coercive} also radially unbounded. Let $\Delta_m$ denote the unit simplex in $\mathbb{R}^m$ and let $$\eta=  \sup_{x\in\Omega,\lambda\in\Delta_m} V\left(\sum_{i=1}^{m}{\lambda_i f_i(x)}\right)-V(x).$$

We claim that $\eta<0$. This is because of (\ref{eq:Lyap.decrease}) and the fact that the above supremum is achieved as the objective functions is continuous and the feasible set is compact. Hence, the sequence $\{V(x_k)\}$ decreases in each step by at least $|\eta|$ and hence must go to $-\infty$. This however contradicts positivity of $V$ on $\Omega$.
\end{proof}

%AAA: I removed this remark because I found it confusing and it also seemed to have a technical issue: We know extistence of a convex Lyapunov function for the corners if the JSR<1. But JSR<1 is said earlier in the paper to be the same as stability of the convex hull. This makes the remark circular.

%\begin{remark}
%An immediate consequence of Proposition~\ref{prop:convex.lyap} is the well-known fact that the switched linear system (\ref{eq:switched.linear.system}) is asymptotically stable if and only if the system that at each time step switched only among one the corner matrices $A_1,\ldots, A_m$ (and not a matrix in their convex hull) has the same property. Indeed, 
%
%It is a direct consequence of the theorem above that a \emph{linear} switched system defined  by a finite number of matrices (i.e., at each time step, one of these matrices is applied to the system) is stable if and only if the switched system defined by the convex hull of the set of matrices is stable.  Indeed, it is well known that the former system is stable if and only if there exists a common convex Lyapunov function for it (see Theorem \ref{thm-rotastrang}), which directly implies that the convex hull is also stable.
%\end{remark}t

\subsection{Computing regions of attraction for switched nonlinear systems} In this section, we consider the switched nonlinear system in (\ref{eq:nl.systems.again}), where $f_1,\ldots,f_m$ are polynomials. It is quite common in this case for the system to not be globally stable but yet to have a locally attractive equilibrium point.
%As before, we take the equilibrium point of the system to be at the origin without loss of generality. 
Under the assumption that $\rho(A_1,\ldots,A_m)<1$, where $A_1,\ldots,A_m$ are the matrices associated with the linearizations of $f_1,\ldots,f_m$ around the origin, we design an algorithm based on semidefinite programming that provably finds a full-dimensional inner approximation to the region of attraction of the nonlinear switched system. Note that if the origin of (\ref{eq:nl.systems.again}) is locally asymptotically stable, then we must have $\rho(A_1,\ldots,A_m) \leq 1.$ The only case to remain is the boundary case $\rho(A_1,\ldots,A_m)=1,$ which is left for our future work.

% Hence, all cases except when $\rho(A_1,\ldots,A_m)=1$ are covered by our algorithm.

Our procedure for finding the region of attraction will have two steps:
\begin{enumerate}[(i)]
\item Use SDP to find a common sos-convex Lyapunov function $V$ for the linearizations of $f_1,\ldots, f_m$ around the origin; i.e., find a positive definite sos-convex form $V(x)$ such that $V(x)-V(A_i x)$ is sos and positive definite for $i=1,\ldots,m.$ Existence of such a function is guaranteed by Theorem \ref{thm:existence-sos-convex-poly}, which was the main result of Section \ref{section:linear}.
\item Find a scalar $\beta>0$ such that $$\forall x \neq 0, V(x)\leq \beta \Rightarrow V(f_i(x))<V(x), \text{ for } i=1,\ldots,m.$$ We will prove that semidefinite programming can find such a $\beta$ in finite time and certify the above implication algebraically. 

%such a $\beta$ exists 

% and can be obtained using bisection and semidefinite programming. Moreover, the SDP will return an algebraic certificate of the above implication.
\end{enumerate}

Once this procedure is carried out, the set $\{x\in\mathbb{R}^n|~V(x)\leq \beta\}$ is guaranteed to be a subset of the region of attraction. Implementation of step (ii) requires the reader to be reminded of the following fundamental theorem in algebraic geometry. Recall that a basic semialgebraic set is a set defined by a finite number of polynomial inequalities and equations.

\begin{theorem}[Stengle's Positivstellensatz~\cite{stengle1974nullstellensatz}]\label{thm:stengle}
	The basic semialgebraic set $$S=\{x\in \mathbb{R}^n |~ g_1(x)\geq 0,\ldots, g_m(x)\geq 0, h_1(x)=0,\ldots, h_s(x)=0\}$$ is empty if and only if there exist polynomials $t_1(x),\ldots,t_s(x)$ and sum of squares polynomials $\{\sigma_{a_1\ldots a_m}|~ (a_1,\ldots,a_m)\in \{0,1\}^{m}\}$ such that
	\begin{align*}
	-1=\sum_{j=1}^s t_j(x) h_j(x)+\sum_{a_1,\ldots,a_m \in \{0,1\}^m} \sigma_{a_1 \ldots a_m}(x) \prod_{i=1}^m g_i^{a_i}(x).
	\end{align*}
\end{theorem}

\begin{theorem}\label{thm:nl.beta.guarantee}
Consider the switched nonlinear system in (\ref{eq:nl.systems.again}), where $f_1,\ldots,f_m: \mathbb{R}^n \rightarrow \mathbb{R}^n$ are polynomials. Let $$f_1^l(x)=A_1x,\ldots,f_m^l(x)=A_mx$$ be the linearizations of $f_1,\ldots,f_m$ around zero, and suppose $\rho(A_1,\ldots,A_m)<1$. Let $y \in \mathbb{R}$ be a new variable. Then, there exist an sos-convex positive definite form $V(x)$, a scalar $\beta>0$, a polynomial $t$, and sum of squares polynomials $$\{\sigma_{a_0\ldots a_m}|~ (a_0,\ldots,a_m)\in \{0,1\}^{m+1}\}$$ such that 
\rmj{	\begin{equation}\label{eq:stengle.lyap}
	\begin{aligned}
	-1=&t(x,y)(\sum_{i=1}^n x_i^2\cdot y-1)\\
	&+\sum_{a_0,\ldots,a_m \in \{0,1\}^{m+1}} \sigma_{a_0\ldots a_m}(x,y) (\beta-V(x))^{a_0} \prod_{i=1}^m (V(f_i(x))-V(x))^{a_i}.
	\end{aligned}
	\end{equation}}
	%	-1=&h(x,y)\cdot (\sum_{i=1}^n x_i^2\cdot y-1)+\sigma_0(x,y)\\
	%	&+\sum_{i=1}^m \sigma_i(x,y) (V(f_i(x))-V(x)) +\sigma_{m+1}(x,y)\cdot (\beta -V(x))\\
	%	&+\sum_{1 \leq i<j \leq m} \sigma_{ij}(x,y) (V(f_i(x))-V(x))(V(f_j(x))-V(x))\\&+\ldots+\sigma_{12\ldots m+1}(x,y) \prod_{i=1}^{m} (V(f_i(x))-V(x)) \cdot (\beta-V(x)).
	Conversely, if (\ref{eq:stengle.lyap}) holds, then the switched nonlinear system in (\ref{eq:nl.systems.again}) is locally asymptotically stable under arbitrary switching and the set $$\{x \in \mathbb{R}^n|~ V(x)\leq \beta\}$$ is a subset of the region of attraction.
\end{theorem}

\begin{proof}
	We start with the converse as it is the easier direction to prove. Note that if (\ref{eq:stengle.lyap}) holds for some sos-convex positive definite form $V(x)$ and some scalar $\beta>0$, then the set 
	\begin{align}\label{eq:empty.set}
	\{(x,y)\in\mathbb{R}^{n+1}|~ V(x) \leq \beta, (\sum_{i=1}^n x_i^2)\cdot y=1, V(f_i(x))-V(x)\geq 0, i=1,\ldots,m \}
	\end{align}
	is empty. Indeed, if there was a point $(\hat{x},\hat{y})$ in this set, then plugging it into (\ref{eq:stengle.lyap}) would give a contradiction as the right hand side would evaluate to a nonnegative real number.
 We observe that emptiness of (\ref{eq:empty.set}) is equivalent to emptiness of 
	\begin{align}\label{eq:empty.set.2}
	\{x \in \mathbb{R}^n|~ V(x) \leq \beta, x \neq 0, V(x)-V(f_i(x))\leq 0, i=1,\ldots,m\},
	\end{align}
	which in turn implies that $$\forall x \neq 0, V(x)\leq \beta \Rightarrow V(f_i(x))<V(x), \text{ for } i=1,\ldots,m.$$ From Proposition \ref{prop:convex.lyap} (part (ii)), it follows that the switched nonlinear system in (\ref{eq:nl.systems.again}) is locally asymptotically stable under arbitrary switching and that the set $$\{ x \in \mathbb{R}^n|~ V(x)\leq \beta\}$$ is a subset of the region of attraction.

	We now show the opposite direction. Since $\rho(A_1,\ldots,A_m)<1$, we know from Theorem \ref{thm:existence-sos-convex-poly} that there exists a positive definite sos-convex form $V$ of some even degree $r$ such that 
	\begin{align}\label{eq-proof-roa}
	V(A_ix)<V(x),~\forall x\neq 0, \text{ and } i=1,\ldots,m.
	\end{align}
	We prove that there exists a scalar $\beta>0$ such that
	\begin{align}\label{eq:roa.proof}
	\forall x \neq 0, V(x)\leq \beta \Rightarrow V(f_i(x))<V(x), \text{ for } i=1,\ldots,m,
	\end{align}
	by considering the Taylor expansion of $V$ around the origin. As the maps $f_i,~ i=1,\ldots,m, $ are twice differentiable, we have
	 \rmj{\begin{eqnarray} V(f_i(x))- V(x)&=& V(A_ix + O(||x||^2))- V(x) \nonumber \\&=& V(A_ix)+   g(x) - V(x),\label{eq:V.f}\end{eqnarray}
	where $g(x)$ is $O(||x||^{r+1})$, that is, there exist $\delta_i>0$ and $K_i>0$ such that if $||x||<\delta_i$, 
	 \begin{align}\label{eq:defn.big.O}
	 |g(x)| \leq K_i ||x||^{r+1}.
	 \end{align} 
	
	Also, n}ote that $V(A_ix)-V(x)$ is a degree-$r$ form which is negative definite. Hence, if we define $$\lambda_i \mathrel{\mathop{:}}=-\frac{1}{2}\min_{||x||=1} (V(A_ix)-V(x)),$$ we have $\lambda_i>0$ and 
	 \begin{align} \label{eq:neg.def}
	 V(A_ix)-V(x)<-\lambda_i ||x||^r.
	 \end{align} 
	 Let $\epsilon_i=\min(\delta_i, \frac{\lambda_i}{K_i})$ and note that $\epsilon_i>0.$ For any nonzero $x$ such that $||x||\leq \epsilon_i$, we have
	 \begin{align*}
	 V(f_i(x))-V(x)&< -\lambda_i||x||^r+O(||x||^{r+1})\\
	 &\leq -\lambda_i ||x||^r+K_i||x||^{r+1}\\
	 &\leq 0,
	 \end{align*}
	 where the first inequality follows from (\ref{eq:V.f}) and (\ref{eq:neg.def}), the second from (\ref{eq:defn.big.O}) as $||x||\leq \delta$, and the third from the fact that $||x||\leq \frac{\lambda_i}{K_i}.$ By compactness of the sublevel sets of $V$ and homogeneity of $V$, there exists $\beta_i>0$ such that $V(x)\leq \beta_i \Rightarrow ||x||<\epsilon_i$. Taking $\beta=\min_{i=1,\ldots,m} \beta_i$ concludes the proof of (\ref{eq:roa.proof}).

Now observe that the statement in (\ref{eq:roa.proof}) implies that the set in (\ref{eq:empty.set.2}) is empty. This is equivalent to the set in (\ref{eq:empty.set}) being empty as noted previously. From Theorem \ref{thm:stengle}, this implies that there exist  a polynomial $t$, and sum of squares polynomials $\{\sigma_{a_0\ldots a_m}|~ (a_0,\ldots,a_m)\in \{0,1\}^{m+1}\}$ such that the algebraic identity in (\ref{eq:stengle.lyap}) holds.	
\end{proof}

Theorem \ref{thm:nl.beta.guarantee} gives rise to a hierarchy of semidefinite programs whose $r^{th}$ level involves searching for a polynomial $t$ and sum of squares polynomials $$\{\sigma_{a_0\ldots a_m}|~ (a_0,\ldots,a_m) \in \{0,1\}^{m+1} \}$$ of degree less than or equal to $2r$ that satisfy the algebraic identity in (\ref{eq:stengle.lyap}) (note that $V$ is fixed here). For fixed $r$, one can obtain the largest $\beta$ for which (\ref{eq:stengle.lyap}) is feasible by doing bisection on $\beta$. If the number $m$ of maps and the level $r$ of the hierarchy are fixed, one can check that the size of the resulting SDP is polynomial in the number of variables $n.$

We also remark that this SDP-based procedure terminates in finite time with a full-dimensional estimate of the ROA. Indeed, one can bound the degrees of the polynomials $t_j, j=1,\ldots,r$ and $\sigma_{a_1\ldots a_m}, a_1,\ldots,a_m \in \{0,1\}^m$ in Theorem \ref{thm:stengle} by quantities that only depend on the degree of the polynomials $h_i$ and $g_i$, $m$, $n$, and $s$ (see \cite{lombardi2014elementary} for the precise bound). So in theory, if we fix the degree of the polynomials $t$ and $\sigma_{a_1\ldots a_m}, a_1,\ldots,a_m \in \{0,1\}^m$ in (\ref{eq:stengle.lyap}) to that bound, start with any $\beta>0$, and halve $\beta$ when the SDP is infeasible, then the procedure will terminate in finite time with a positive $\beta$ for which the SDP is feasible. The bounds in \cite{lombardi2014elementary} are too large however to be practical and hence our remark here is of theoretical interest only. In practice, we observe that the first few levels of the hierarchy are sufficient to obtain a full-dimensional estimate of the ROA.

\subsection{Examples: ROA computation for nonlinear switched systems}\label{subsec:examples}
We give two examples of the ideas we have seen so far for local stability analysis.
%Our technique also allows for computation of inner approximations to regions of attraction when the switched nonlinear system is not globally stable.  We show this on two examples.
\begin{example}\label{ex:convex.lyap}
Let us revisit the system (\ref{eq:ex.nonconvex.fails}) of Example \ref{ex:nonconvex.fails}.  We claim that the function $$W(x)=x_1^2+x_2^2, $$ which is convex, is a common Lyapunov function for $f_1,f_2$ on the set $$S=\{x \in \mathbb{R}^n|~||(x_1,x_2)^T||_\infty\leq 1\}.$$ Indeed, for $i=1,2,$ and nonzero $x\in\mathcal{S},$
\begin{eqnarray}\nonumber W(f_i(x))&=&x_1^2x_2^2 \\ \nonumber &< & x_1^2+x_2^2\\\nonumber &=& W(x).\end{eqnarray} Moreover, $S$ is an invariant set for $f_1$ and $f_2$. Hence, for the system (\ref{eq:ex.nonconvex.fails}), the set $S$ is part of the region of attraction of the origin under arbitrary switching.
\end{example}

%[To do: AAA - give an example of a *polynomial* switched system and find ROA under arbitrary switching.
%Can do one where no polynomial Lyap of degree less than 10 finds \emph{any} nontrivial ROA, but a convex one of degree ten
%finds ROA. (Essentially its linearization should be pair of matrices that do not admit a common poly Lyap of degree less than 10. We already have this in our SICON paper.)]
We now give an example where quadratic Lyapunov functions do not suffice for a proof of local stability and our SDP procedure is carried out in full.

\begin{example}\label{ex:nonlin.roa}
\begin{figure} [h]
	\centering
\centering \scalebox{0.33} {\includegraphics{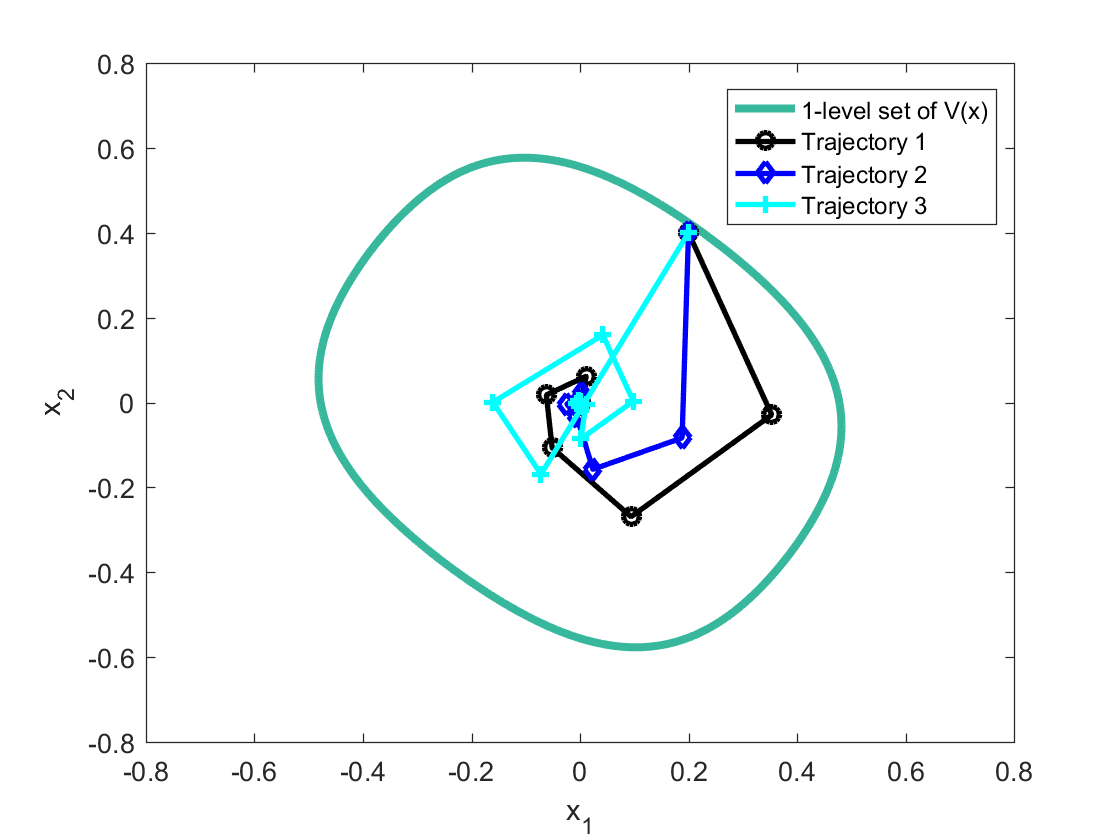}} 
\caption{The 1-level set of $V$ in (\ref{eq:V.ex}) as well as three possible trajectories of the nonlinear switched system in Example \ref{ex:nonlin.roa} starting from the same initial condition.}\label{fig:traj}
\end{figure}
\end{example}
Consider the dynamical system in (\ref{eq:nl.systems.again}), with $m=n=2$ and
\begin{align*}
f_1(x_1,x_2)&=\begin{pmatrix}-\frac14 x_1-\frac14 x_2+\frac15 x_1^2\\ -x_1+\frac{1}{10}x_1x_2 \end{pmatrix},\\
f_2(x_1,x_2)&=\begin{pmatrix} \frac34 x_1+\frac34 x_2-\frac{1}{10}x_1x_2\\ -\frac12 x_1+\frac14 x_2 \end{pmatrix}.
\end{align*}
The linearizations of $f_1$ and $f_2$ at $(x_1,x_2)=0$ are given by $f_{1}^l(x)=A_1x$ and $f_2^l (x)=A_2x$, where 
$$A_1=\begin{pmatrix} -1/4&-1/4\\ -1 & 0\end{pmatrix} \text{ and } 
A_2=\begin{pmatrix}3/4 & 3/4 \\ -1/2 & 1/4 \end{pmatrix}.$$
One can check that these matrices do not admit a common quadratic Lyapunov function. We will consequently be searching for polynomials of higher order. In this case, imposing convexity becomes essential as it is no longer implied by nonnegativity of the polynomial. Using the parser YALMIP\cite{yalmip} and the SDP solver MOSEK\cite{mosek2015}, we look for a quartic form $V$ satisfying the sos conditions of Theorem \ref{thm:existence-sos-convex-poly}. Our SDP solver returns the sos-convex form
\begin{align}\label{eq:V.ex}
V(x_1,x_2)=19.14x_1^4+10.57x_1^3x_2+47.88x_1^2x_2^2+16.47x_1x_2^3+10.49x_2^4.
\end{align}
This implies that $\rho(A_1,A_2)<1.$ By solving a second SDP, one can find a polynomial $t$ of degree $\leq 4$ and sos polynomials $\sigma_0$, $\sigma_1$, $\sigma_2$, $\sigma_3$, $\sigma_{12}, \sigma_{23},\sigma_{13}$ and $\sigma_{123}$ of degree $\leq 4$ that satisfy (\ref{eq:stengle.lyap}) with $\beta=1.$ From the ``easy'' direction of Theorem \ref{thm:nl.beta.guarantee}, this implies that the set $$\{x \in \mathbb{R}^n|~V(x)\leq 1\}$$ is part of the region of attraction of the nonlinear switched system given by $f_1$ and $f_2$. This is illustrated in Figure \ref{fig:traj}, where we have plotted the 1-level set of $V$, and three possible trajectories of our switched dynamical system. These trajectories are generated by the dynamics $x_{k+1}=\lambda f_1(x_k)+(1-\lambda)f_2(x_k),$ where $x_0=(0.2,0.4)$ for all three trajectories and $\lambda$ is picked uniformly at random in $[0,1]$ at each iteration. As can be seen, all trajectories flow to the origin as predicted by the theory. 

\section{Conclusions and extensions to multiple Lyapunov functions}\label{sec:future}

In this paper, we introduced the concept of sos-convex Lyapunov functions for stability analysis of switched linear and nonlinear systems. For switched linear systems, we proved a converse Lyapunov theorem on guaranteed existence of sos-convex Lyapunov functions. We further showed that the degree of a convex polynomial Lyapunov function can be arbitrarily higher than the degree of a non-convex one. For switched nonlinear systems, we showed that sos-convex Lyapunov functions allow for computation of regions of attraction under arbitrary switching, while non-convex Lyapunov functions in general do not.

Our work can be extended in at least two different directions. The first direction concerns the computation of the region of attraction of the nonlinear switched system in (\ref{eq:nl.systems.again}) when the joint spectral radius of the matrices associated to the linearizations of $f_1\ldots,f_m$ is exactly equal to one. In this scenario, the assumption of Theorem \ref{thm:nl.beta.guarantee} is violated. Nevertheless, one can directly search for an sos-convex polynomial $V$, a scalar $\beta>0$, a polynomial $t$, and sos polynomials $\sigma_{a_0\ldots a_m}$ satisfying (\ref{eq:stengle.lyap}) to have a certificate that the $\beta$-sublevel set of $V$ is in the ROA of the origin. The problem with this approach however is that the coefficients of $V$ and $\sigma_{a_0\ldots a_m}$ are all decision variables and their multiplication leads to a nonconvex constraint. A principled way of getting around this issue with convex relaxations is left for our future work. 

The second direction is motivated by scalability issues encountered when solving semidefinite programs arising from sos constraints on high-degree polynomials. In general, it is more efficient to work with \emph{multiple} low-degree sos-convex Lyapunov functions as opposed to a single one of high degree. This is because the underlying semidefinite program will end up having semidefinite constraints on much smaller matrices (though possibly a higher number of them). Nevertheless this trade-off is almost always computationally beneficial for interior point solvers. 

A systematic approach for searching for multiple Lyapunov functions that together imply stability of a switched linear system has been proposed in \cite{ajpr-sicon}. If the switched system is defined by $x_{k+1}=A_i x_k, ~i=1,\ldots,m,$ and our candidate Lyapunov functions are $V_1,\ldots,V_r,$ the works in \cite{ajpr-sicon} and \cite{jungers2017characterization} completely characterize all collections of Lyapunov inequalities of the type $$\{V_j(A_ix) < V_k(x)\}$$ that prove stability. This characterization is based on the concept of \emph{path-complete graphs} (see~\cite[Definition 2.2]{ajpr-sicon}), which is a notion that relates to the theory of finite automata and languages. In our future work, we would like to extend this theory to cover nonlinear switched systems. In this setting, the property of convexity needs to be carefully incorporated, as the current paper has demonstrated. More precisely, we would like to understand which path-complete paths give rise to a common \emph{convex} Lyapunov function, assuming that the nodes of the graph are all associated with convex Lyapunov functions. The proposition below provides a large family of such graphs, though we suspect that there must be others. In the reader's interest, we present the proposition in a self-contained fashion with no mention to the terminology of path-complete graphs. The common convex Lyapunov function obtained here will be a pointwise maximum of convex functions. A complete study of the more general question above would likely need to extend the ideas in \cite[Section III]{philippe2017path} and \cite[Section IV]{jungers2017characterization}. 

For simplicity, we state the proposition below for global asymptotic stability. The analogous statement for local asymptotic stability is simple to derive (similarly to what was done in Proposition \ref{prop:convex.lyap}).

\begin{proposition}\label{cor:max.of.quadratics}
Consider the nonlinear switched system in (\ref{eq:switched.nonlinear.system.ex}) defined by continuous maps $f_1,\ldots,f_m:\mathbb{R}^n \rightarrow \mathbb{R}^n$. If there exist $K$ convex Lyapunov functions $V_1,\ldots,V_K:\mathbb{R}^n\rightarrow\mathbb{R}$ that satisfy $V_i(0)=0$, $V_i(x)>0$ for all $x\neq 0$, $\forall i\in \{1,\ldots,K\},$ and
\begin{eqnarray}\label{eq:max.quadratics.LMIs}
\forall (i,j) \in\{1,\ldots,m\}\times \{1,\ldots,K\},\ \exists k\in\{1,\ldots,K\}\
\nonumber
\\
\mbox{such that}\quad V_j(f_i(x))< V_k(x), \quad \forall x \neq 0,
\end{eqnarray}
then the origin is globally asymptotically stable under arbitrary switching. Moreover, if these constraints are satisfied, then the convex function 
$$W(x)\mathrel{\mathop{:}}=\max \{V_1(x),\ldots ,V_K(x)\}$$ is a common Lyapunov function for $f_1,\ldots, f_m$.
%
%
%
%Consider the set $\{f_1(x_k),\ldots,f_m(x_k)\}$ and the
%associated switched system in
%(\ref{eq:switched.nonlinear.system}). If there exist $L$
%positive homogeneous functions $V_j$ satisfying the following inequalities:
%\begin{eqnarray}\label{eq:max.quadratics.LMIs}
%\forall (i)\in\{1,\ldots,m\},\forall (j)\in\{1,\ldots,L\},\ \exists k\in\{1,\ldots,m\}\
%\nonumber
%\\
%\mbox{such that}\quad \forall x, \quad V_j(f_i(x))< V_k(x),
%\end{eqnarray}
%then the system is asymptotically stable.  Moreover, the pointwise
%maximum
%$$\max_j\{V_j(x)\}$$ is a common Lyapunov function.  
\end{proposition}  
\begin{proof}
It suffices to show the latter claim because the former would then follow from Proposition \ref{prop:convex.lyap}, part (i), as it is clear that $W$ so constructed is positive definite and convex.
Let $i \in \{1,\ldots,m\}$ be fixed. From (\ref{eq:max.quadratics.LMIs}), for any $j \in \{1,\ldots,K\},$ there exists $k \in \{1,\ldots,K\}$ such that $$V_j(f_i(x))<V_k(x), \forall x\neq 0.$$ As $W$ is the pointwise maximum of $V_k, k=1,\ldots,K$, it follows that $$V_j(f_i(x))<W(x), \forall x\neq 0 \text{ and } \forall j\in \{1,\ldots,K\}.$$
Hence $W(f_i(x))< W(x), \forall x\neq 0.$
%
%Consider $x_t$ at any time $t\geq 0.$ Suppose that $\sigma(t)=i.$
%By equation \eqref{eq:max.quadratics.LMIs}, for all $j$ there is a $k$ such that $$V_j(f_i(x))\leq V_k(x) \leq \max_k{ V_k(x)}.$$
%Summarizing, for all $j,$ $V_j(f_i(x))\leq\max_k{ V_k(x)}$ and the proof is done.
\end{proof}

In the case where $f_1,\ldots,f_m$ are polynomials, and $V_1,\ldots,V_k$ are parametrized as sos-convex polynomials, the search for $W$ can be carried out by semidefinite programming after replacing the inequalities in (\ref{eq:max.quadratics.LMIs}) with their sos counterparts. Note that the above proposition does not give just one way of formulating such an SDP, but rather $K^{m^2}$ of them. Indeed, for any fixed pair $(i,j)$, there are $K$ choices for the index $k$. In the language of \cite{ajpr-sicon}, each of these SDPs corresponds to a particular path-complete graph and its feasibility provides a proof of stability. 

\section*{Acknowledgments}
%Amir Ali Ahmadi is supported by a Goldstine Fellowship at IBM Research. Rapha\"el Jungers is supported by the Communaut\'e fran\c caise de Belgique - Actions de Recherche
%Concert\'ees, and by the Belgian Programme on Interuniversity Attraction Poles initiated
%by the Belgian Federal Science Policy Office. 
The \rmj{authors are} thankful to Alexandre Megretski for insightful discussions around convex Lyapunov functions.
%R.J. is an F.R.S.-FNRS Research Associate.  His work is supported by the Communaut\'e fran\c caise de Belgique - Actions de Recherche
%Concert\'ees, and by the Belgian Programme on Interuniversity Attraction Poles initiated
%by the Belgian Federal Science Policy Office.

%If switched stability of a \emph{nonlinear} system of the form (\ref{eq:switched.nonlinear.system})-(\ref{eq:ftilda=conv}) is %of concern, then existence of a common convex Lyapunov function is no longer necessary (even when $m=1$). 

\bibliographystyle{plain}
\bibliography{pablo_amirali}

\end{document}